\documentclass[a4paper]{amsart}

\usepackage[T1]{fontenc}
\usepackage[utf8]{inputenc}

\usepackage{lmodern}

\usepackage[english]{babel}
\usepackage{csquotes}
\usepackage{amsmath}
\usepackage{amssymb}
\usepackage{amsthm}
\usepackage{mathtools}
\usepackage{microtype}
\usepackage[unicode,hyperfootnotes=false]{hyperref}
\hypersetup{
pdftitle={Large sets without Fourier restriction theorems},
pdfauthor={Constantin Bilz},
colorlinks=true,
citecolor=[rgb]{0,0,0.5},
linkcolor=[rgb]{0,0,0.5},
urlcolor=[rgb]{0,0,0.75},
pdfstartview=FitH,
pdfpagemode=UseNone,
pdfdisplaydoctitle=true,
pdflang=en-US
}
\hypersetup{breaklinks=true}

\usepackage[style=alphabetic]{biblatex}
\usepackage{enumitem}
\usepackage{esint}
\usepackage[capitalise]{cleveref}
\crefname{equation}{}{}
\crefname{subsection}{Subsection}{Subsections}
\crefname{subsubsection}{Subsubsection}{Subsubsections}
\usepackage[shortcuts]{extdash}

\newcommand{\R}{\mathbb R}
\DeclareMathOperator{\diam}{diam}
\newcommand{\rexp}{p_{\textnormal{res}}}
\newcommand{\dimH}{\dim_{H}}

\newtheorem{lemma}{Lemma}
\newtheorem{proposition}[lemma]{Proposition}
\newtheorem{theorem}[lemma]{Theorem}
\newtheorem{corollary}[lemma]{Corollary}
\theoremstyle{remark}
\newtheorem{remark}[lemma]{Remark}

\title{Large sets without Fourier restriction theorems}
\author{Constantin Bilz}
\address{Constantin Bilz,
School of Mathematics,
University of Birmingham,
Edgbaston,
Birmingham,
B15 2TT,
England}
\email{\href{mailto:C.Bilz@pgr.bham.ac.uk}{C.Bilz@pgr.bham.ac.uk}}
\date{\today}


\begin{document}

\begin{abstract}
We construct a function that lies in $L^p(\R^d)$ for every $p \in (1,\infty]$ and whose Fourier transform has no Lebesgue points in a Cantor set of full Hausdorff dimension.
We apply Kova\v{c}'s maximal restriction principle to show that the same full\-/dimensional set is avoided by any Borel measure satisfying a nontrivial Fourier restriction theorem.
As a consequence of a near\-/optimal fractal restriction theorem of {\L}aba and Wang, we hence prove a lack of valid relations between the Hausdorff dimension of a set and the range of possible Fourier restriction exponents for measures supported in the set.
\end{abstract}

\maketitle


\section{Introduction}

It is a fundamental fact that the Fourier transform of an $L^1(\R^d)$ function is uniformly continuous.
We complement this with the following main result.

\begin{theorem}
\label{t}
There exists a function in $\bigcap_{p \in (1,\infty]} L^p(\R^d)$ whose Fourier transform has no Lebesgue points in some compact set of Hausdorff dimension $d$.
\end{theorem}

We say that a point $x \in \R^d$ is a \emph{Lebesgue point} of a function $g \colon \R^d \to \mathbb C$ if there exists a number $c \in \mathbb C$ such that $r^{-d} \int_{\{|y|<r\}} |g(x-y)-c| \, dy \to 0$ as $r \to 0$.
The set of non\-/Lebesgue points of a function is called its \emph{non\-/Lebesgue set}.
The Lebesgue differentiation theorem states that non\-/Lebesgue sets of locally integrable functions have Lebesgue measure zero.
\cref{t} shows that this cannot be sharpened in terms of Hausdorff dimension for the class of Fourier transforms of $L^p(\R^d)$ functions when $p>1$.

Our interest in this problem stems from a measure\-/theoretic perspective on Fourier restriction theory which was recently introduced by Müller, Ricci and Wright~\cite{MRW19}.
They asked for a pointwise interpretation of restrictions of Fourier transforms.
Due to subsequent work of Kova\v{c}~\cite{Kov19} it is known that, under fairly general assumptions, sets of divergence of local averages of Fourier transforms are avoided by measures permitting Fourier restriction theorems.
Since many such measures are known, this is a strong structural condition on non\-/Lebesgue sets of Fourier transforms.
\cref{t} shows that these sets can nevertheless be large in a metric sense.

This observation has implications for restriction theory.
Using \cref{t} and Kova\v{c}'s result we will prove the following result which asserts the existence of a set of full Hausdorff dimension and without nontrivial restriction theorems.
Let $\mathcal S(\R^d)$ denote the Schwartz space.
\begin{corollary}
\label{c:endpt}
There exists a compact subset $E$ of $\R^d$ such that $E$ has Hausdorff dimension~$d$ and for any Borel measure $\mu$ on $\R^d$ with $\mu(E)>0$ and for any $p \in (1,2]$ and any $q \in [1,\infty]$ it holds that
\[
\sup_{f \in \mathcal S(\R^d)}
\frac{\|\hat f\|_{L^q(\mu)}}{\|f\|_{L^p(\R^d)}}
=
\infty.
\]
\end{corollary}
We will further strengthen this corollary by showing a lack of valid relations between the Hausdorff dimension of a set and the supremum of the range of exponents $p$ for which there are $L^p(\R^d)$\=/based restriction theorems on that set.
More precisely, the only relation that holds between these numbers is a well\-/known energy\-/theoretic inequality, see \cref{c:full} below.
For the proof of this, we will rely on a recent fractal restriction theorem of {\L}aba and Wang~\cite{LW18} which is near\-/optimal with respect to that inequality.

\subsection{Restriction theorems and convergence of averages}

One classical way of proving restriction theorems is the Tomas--Stein argument~\cite{Tom75}.
In its general endpoint form due to Mockenhaupt~\cite{Moc00}, Mitsis~\cite{Mit02} and Bak and Seeger~\cite{BS11} it implies the following:
if $\mu$ is a finite Borel measure on $\R^d$ satisfying the pointwise Fourier decay condition
\begin{equation}
\label{e:fdec}
\sup_{\xi \in \R^d} |\xi|^{\beta/2} |\hat \mu(\xi)| < \infty
\end{equation}
for some $\beta \in [0,d)$, then the restriction inequality
\begin{equation}
\label{e:res}
\|\hat f\|_{L^q(\mu)}
\leq
C
\|f\|_{L^p(\R^d)}
\end{equation}
holds for any $p \in [1, 4d/(4d-\beta)]$ and $q=2$ and for any function $f \in \mathcal S(\R^d)$.
The constant $C$ above is independent of $f$.
The full Mockenhaupt--Mitsis--Bak--Seeger theorem contains an additional dimensionality condition that leads to a larger range of restriction exponents $p$ in many situations.
The exponent $4d/(4d-\beta)$ that we give above corresponds to the dimensionality that is implied by \cref{e:fdec}, see e.g.\ \cite[Corollary~3.1]{Mit02}.
Regarding the sharpness of the Mockenhaupt--Mitsis--Bak--Seeger theorem see Hambrook and {\L}aba \cite{HL16} and previous works cited therein.

The supremum of the range of rates $\beta \in [0,d)$ for which \cref{e:fdec} holds is commonly called the \emph{Fourier dimension} of the finite nonzero Borel measure $\mu$.
The \emph{Fourier dimension} of a subset of $\R^d$ is the supremum of the set of Fourier dimensions of all finite nonzero Borel measures that are compactly supported in that set.

Hence, the Mockenhaupt--Mitsis--Bak--Seeger theorem yields a nontrivial restriction theorem for any Borel measure or set of strictly positive Fourier dimension.
This applies to classical examples such as submanifolds that are curved in an appropriate sense, but it also applies to various types of fractals.

The restriction theory of submanifolds of $\R^d$ was initiated by Stein in 1967, see~\cite[p.~374]{Ste93}.
Typical objects of study include hypersurfaces such as the sphere, the paraboloid and the cone, as well as lower\-/dimensional submanifolds and curves.
Despite significant progress, Stein's restriction conjecture~\cite{Ste79} remains unresolved in most cases.
The methods employed in this subject reach significantly beyond the Tomas--Stein type arguments alluded to above.
We refer the reader to the recent survey~\cite{Sto19} and the references therein.

In order to demonstrate the wide applicability of the Mockenhaupt--Mitsis--Bak--Seeger theorem to fractals we mention a few interrelated classes of deterministic and random fractal sets that have a positive Fourier dimension, with selected references:
\begin{itemize}
\item images of stochastic processes~\cite{Kah85} and random diffeomorphisms~\cite{Eks16},
\item certain sets arising from Diophantine approximation~\cite{Kau81,JS16},
\item limit sets of certain nonlinear group actions~\cite{BD17} and
\item various constructions based on Cantor sets~\cite{Sal51,LW18}.
\end{itemize}
For further information on these matters, we refer the reader to the survey \cite{Lab14}.

If the restriction inequality \cref{e:res} holds, then the Fourier transform on $\mathcal S(\R^d)$ extends to a bounded \emph{restriction operator} $\mathcal R_\mu \colon L^p(\R^d) \to L^q(\mu)$.
In the case of a singular measure $\mu$, this operator is often regarded as a natural way to assign values $\mu$\=/almost everywhere to the Fourier transform of an $L^p(\R^d)$ function.
Indeed, \cref{e:res} readily implies that for any $f \in L^p(\R)$ there exists a sequence of radii $r_n \to 0$ such that
\begin{equation*}
\lim_{n \to \infty} \fint_{\{|y| < r_n\}} \hat f(x-y) \, dy = \mathcal R_\mu f(x)
\qquad
\text{$\mu$-a.e.}
\end{equation*}
where we use the average integral notation
\[
\fint_A h(y) \, dy
=
\frac1{|A|}
\int_A h(y) \, dy
\]
and $|A|$ is the Lebesgue measure.

Recently Müller, Ricci and Wright~\cite{MRW19} proved a \emph{maximal} restriction theorem for planar curves and used it to strengthen the mode of convergence as follows:
\begin{equation}
\label{e:lebR}
\lim_{r \to 0} \fint_{\{|y| < r\}} |\hat f(x-y) - \mathcal R_\mu f(x)| \, dy = 0
\qquad
\text{$\mu$-a.e.}
\end{equation}
when $\mu$ is the affine arc length measure on a smooth planar curve and $f \in L^p(\R^2)$, $1 \leq p < 8/7$.
This means that $\mu$\=/almost every $x \in \R^2$ is a Lebesgue point of $\hat f$ with regularized value $\mathcal R_\mu f(x)$.
Later, Vitturi~\cite{Vit17} obtained a similar result for the surface measure of the sphere $\mathbb S^{d-1} \subseteq \R^d$ and Kova\v{c} and Oliveira e Silva~\cite{KO18} proved a stronger \emph{variational} restriction theorem for spheres.
Then, Kova\v{c}~\cite{Kov19} proved an abstract variational restriction principle that implies the following result.
\begin{theorem}[see {\cite[Remark~3]{Kov19}}]
\label{t:kovac}
If the restriction inequality \cref{e:res} holds for some Borel measure $\mu$ on $\R^d$ and some exponents $p \in [1,2]$ and $q \in (p, \infty)$, then \cref{e:lebR} holds for any $f \in L^{2p/(p+1)}(\R^d)$ and the same measure $\mu$.
\end{theorem}

We note that Kova\v{c}'s theorem allows for more singular averaging kernels in place of the ball averages in \cref{e:lebR}, see also~\cite{Ram19b}, but we will not use it in that generality.

The proof of \cref{t:kovac} relies on the Christ--Kiselev lemma and this is why $q>p$ is assumed.
It is unknown whether this assumption can be removed.
Similarly, it is unknown whether the conclusion can in general be strengthened by replacing $L^{2p/(p+1)}(\R^d)$ with $L^p(\R^d)$.
Here, the seeming inefficiency is due to a reflection argument that is needed in order to obtain a positive (``strong'') maximal inequality from an oscillatory one.
In certain lower dimensional cases, Ramos \cite{Ram19a} used a linearization method to circumvent this issue.

\subsection{Restriction theorems and Hausdorff dimension}

Given a subset $E$ of $\R^d$, we denote by $\rexp(E)$ the supremum of the range of exponents $p \in [1,2]$ for which there exists a Borel measure $\mu$ with $\mu(E) > 0$ such that \cref{e:res} holds for some exponent $q \in [1,\infty]$.
The universal $L^1(\R^d) \to L^\infty(\mu)$ bound implies that $\rexp(E) \geq 1$.
If $E$ has positive Lebesgue measure, then by the Plancherel theorem we have $\rexp(E) = 2$.

An energy integral argument (see e.g.~\cite[Section~2]{Moc00}) shows that $\rexp(E)$ cannot be too large depending on the Hausdorff dimension $\dimH(E)$ and the ambient dimension $d$. Namely, it holds that
\begin{equation*}
\rexp(E) \leq \frac{2d}{2d-\dimH(E)}.
\end{equation*}
We will show that this is the only valid relation between these quantities:
\begin{corollary}
\label{c:full}
Let $\alpha \in [0,d]$ and $p \in [1, 2d/(2d-\alpha)]$.
There exists a compact set $E \subseteq \R^d$ of Hausdorff dimension $\alpha$ and such that $\rexp(E)=p$.
\end{corollary}
The endpoint $p=1$ follows from \cref{c:endpt} and the endpoint $p=2d/(2d-\alpha)$ is the near\-/optimal fractal restriction theorem of {\L}aba and Wang~\cite{LW18}.
This will be enough to prove \cref{c:full}.

We note that the supremum $\rexp(E)$ itself may or may not satisfy \cref{e:res} for suitable $\mu$ and $q$.
The theorem in \cite{LW18} and hence \cref{c:full} do not address this question.

\subsection{Guide to the paper}

\cref{s:osc,s:sinc,s:proof} are dedicated to the proof of \cref{t}.
\cref{s:proofcset} contains the relatively straightforward derivations of \cref{c:endpt,c:full}.

Our proof of \cref{t} relies on a delicate construction based on a Cantor set.
For the reader's convenience, we would like to discuss some features of this construction in an informal way in this subsection.

In \cref{s:osc}, we introduce a family of Cantor sets parameterized by their dissection ratios $\theta_j \in (0,1/2)$, $j \geq 0$.
We establish conditions under which such a Cantor set $E$ is the non\-/Lebesgue set of a certain natural function $g$ and we calculate the inverse Fourier transform $\check g$.
The proof of \cref{t} then comes down to choosing the dissection ratios in such a way that $E$ has full Hausdorff dimension while $\check g$ is $p$\=/integrable for any $p>1$.
The key terms in the $p$\=/integral of $\check g$ are products of cosines resembling
\begin{equation}
\label{e:prodcos}
\prod_{j=0}^{k-1} \cos(\tfrac12\theta_0 \theta_1 \cdots \theta_{j-1} \xi).
\end{equation}

In \cref{s:sinc}, we bound $p$\=/integrals over bounded intervals of slight perturbations of products of cosines with dyadic phases of the form
\begin{equation}
\label{e:prodcosdygap}
\prod_{j \in J} \cos(2^{-j} \xi)
\end{equation}
where $J$ is a set of positive integers.
Our estimate involves some loss depending on the number of components of $J$, but is otherwise near\-/optimal.

The dissection ratios $\theta_j$ that we fix in \cref{s:proof} to complete the proof of \cref{t} have the following essential properties:
\begin{itemize}
\item \emph{The dissection ratios are very close to $1/2$ in an average sense.}
This ensures that the Cantor set has full Hausdorff dimension and it is a prerequisite for an approximation of the products \cref{e:prodcos} by products of cosines with dyadic phases \cref{e:prodcosdygap}.
\item \emph{Infinitely many consecutive pairs of dissection ratios are bounded away from $0$ and $1/2$.}
Under this condition, the Cantor set is the non\-/Lebesgue set of the associated function $g$.
However, the boundedness away from $1/2$ would potentially destroy the similarity between the products \cref{e:prodcos} and \cref{e:prodcosdygap}.
Therefore, we are led to the following condition.
\item \emph{The dissection rates that are not close to $1/2$ are powers of $1/2$.}
Then, in the analysis of \cref{e:prodcos}, these small dissection ratios conveniently translate into gaps in \cref{e:prodcosdygap}, i.e.\ the set $J$ has multiple components.
We choose powers of $1/2$ to exponents that are \emph{large on average} to ameliorate the aforementioned loss depending on the number of components of $J$.
\end{itemize}

\subsection{Further remarks}

\subsubsection*{Connection to the Erd\H{o}s--Kahane theorem}
It is known that for typical values of $\theta$ close to $1/2$, we have $\prod_{j=0}^\infty \cos(\theta^j \xi) = O(|\xi|^{-\epsilon})$ as $|\xi| \to \infty$, with an explicit but rather small $\epsilon > 0$.
This was proved by Erd\H{o}s~\cite{Erd40} and Kahane~\cite{Kah71}, see also the exposition \cite[Section~6]{PSS00}.
Based on this, one could try to prove a weaker version of \cref{t} by using a Cantor set of \emph{constant} dissection ratio $\theta$ and adapting our proof strategy.
This would involve proving estimates for $p$\=/integrals over bounded intervals of the truncated products $\prod_{j=0}^{k-1} \cos(\theta^j \xi)$, $k \geq 0$.
In the case $\theta=1/2$, we easily achieve this by periodicity considerations, see \cref{s:sinc}.
However, due to arithmetic complications that arise when $1/\theta$ is not an integer, this approach does not work for $\theta \in (1/3,1/2)$.
The author believes that an interesting connection to the number\-/theoretic Erd\H{o}s--Kahane theorem above could arise if this obstacle were tackled.

\subsubsection*{Classes of non\-/Lebesgue sets}
The non\-/Lebesgue sets of $L^q(\R^d)$ functions, ${q \in [1, \infty]}$, were characterized by D'yachkov \cite{Dya93} as the $G_{\delta \sigma}$ sets of zero Lebesgue measure.
This class does not depend on $q$.
In contrast, little seems to be known about the smaller class of non\-/Lebesgue sets of Fourier transforms of $L^p(\R^d)$ functions, $p \in (1,2)$, beyond the example of such a set that is provided by \cref{t} and the $p$\-/dependent necessary conditions derived from maximal restriction theorems.

\subsubsection*{Comparison to a theorem of Körner}
\cref{c:full} should be compared to the well\-/known result that for any $\alpha \in [0,d]$ and any $\beta \in [0,\alpha]$ there exists a set of Hausdorff dimension $\alpha$ and Fourier dimension $\beta$.
Körner~\cite{Kor11} proved a stronger version of this statement where the set is further guaranteed to be precisely the support of a measure of Fourier dimension $\beta$.
One may ask whether a similar strengthening of \cref{c:full} is possible.
We note that our synthetic approach to \cref{c:full}, taking the union of two sets with different properties, is unsuitable for this problem.

\subsubsection*{Restriction theorems and Fourier dimension}
Körner's result above is perhaps unsurprising given that Hausdorff dimension is a metric property of a set while Fourier dimension is an arithmetic one.
Similarly, \cref{c:full} is perhaps unsurprising if one accepts that restriction estimates rely on a lack of arithmetic structure of the underlying measure, see e.g.~\cite{Lab14}.
Indeed, the Mockenhaupt--Mitsis--Bak--Seeger theorem shows that nontrivial restriction estimates hold for a measure or a set if its Fourier dimension is positive.
However, a converse of this theorem is to the author's knowledge not available in general.
Hence, we do not know whether Körner's result can be used to prove at least a special case of \cref{c:full} and we do not know whether further relations besides the Mockenhaupt--Mitsis--Bak--Seeger theorem hold between the Fourier dimension and the range of restriction exponents.

\subsection*{Notation} Throughout the paper, the notation $A \lesssim B$ means that $A \leq CB$ holds for a finite positive constant $C$ that is independent of all parameters.

\subsection*{Acknowledgments}

The author would like to express his gratitude to his doctoral supervisors, Diogo Oliveira e Silva and Jonathan Bennett, for their kind support.
He would like to thank Sebastiano Nicolussi Golo for an inspiring discussion at an early stage of this project, as well as Gianmarco Brocchi for helpful discussions.

\section{Cantor sets as non-Lebesgue sets}
\label{s:osc}

Let $\theta_j \in (0,1/2)$, $j \geq 0$, and let $S$ be a set of nonnegative integers.
We write
\[
\Theta_k
=
\theta_0 \theta_1 \cdots \theta_{k-1}.
\]
Using a Cantor set with \emph{dissection ratios} $\theta_j$ we will prove the following result which serves as the starting point for the proof of \cref{t}.

\begin{proposition}
\label{p:criterion}
Let $\theta_j$, $\Theta_k$ and $S$ be as above.
Assume that
\begin{equation}
\label{e:gmean}
\lim_{k \to \infty} \Theta_k^{1/k} = \tfrac12
\end{equation}
and that there exists an $\epsilon > 0$ such that
\begin{equation}
\label{e:goodpairs}
\theta_j,\theta_{j+1} \in (\epsilon, \tfrac12-\epsilon)
\quad
\text{for infinitely many $j \in S$ with $j+1 \in S$.}
\end{equation}
Then there exists a measurable function $g\colon \R^d \to \{-1,0,1\}$ such that
\begin{enumerate}[label=(\roman*)]
\item \label{i:leb}
the non\-/Lebesgue set of $g$ is compact and has Hausdorff dimension $d$ and
\item \label{i:ft}
for any $p>1$, the inverse Fourier transform $\check g$ lies in $\bigcap_{r \in [p,\infty]} L^r(\R^d)$ if
\begin{align*}
\int_0^\infty
\biggl(\sum_{k \in S}
\frac{2^k \Theta_k}
{1 + (1-2\theta_k) \Theta_k |\xi|}
\prod_{j = 0}^{k-1} |{\cos((1-\theta_j) \Theta_j \pi \xi)}|
\biggr)^p
\, d\xi
< \infty.
\end{align*}
\end{enumerate}
\end{proposition}

We simultaneously construct the function $g\colon \R^d \to \{-1,0,1\}$ and the set $E \subseteq \R^d$ that we will show to be the non\-/Lebesgue set of $g$.

Let $|A|$ denote the Lebesgue measure of a Borel set $A \subseteq \R$ and let $m(I)$ denote the midpoint of a bounded nonempty interval $I \subseteq \R$.
We define families $\mathcal W_k$ and $\mathcal B_k$ of \emph{white} and \emph{black intervals} of generations $k=0,1,\ldots$ by the recursion
\begin{align*}
\mathcal W_0 &= \bigl\{\bigl[-\tfrac12,\tfrac12\bigr]\bigr\}, \\
\mathcal B_k &= \{ B_k \text{ open interval} \mid \text{${m(B_k)} = m(W_k)$ \,and\, $|B_k| = (1-2\theta_k) |W_k|$} \\
&\phantom{{}= \{ B_k \text{ open interval} \mid{}}
\text{for some $W_k \in \mathcal W_k$} \}, \\
\mathcal W_{k+1} &= \text{set of connected components of $\bigcup \mathcal W_k \setminus \bigcup \mathcal B_k$.}
\end{align*}
At each generation, one black interval is removed from the middle of each remaining white interval, leaving two white intervals of the next generation.
The remaining white set $\bigcup \mathcal W_k$ is decreasing in $k$.
Its limit as $k \to \infty$ is a Cantor set $E^{(1)} \subseteq \R$ with dissection ratios $\theta_k$:
\begin{equation*}
E^{(1)} = \bigcap_{k=0}^\infty \bigcup \mathcal W_k = [-\tfrac12,\tfrac12] \setminus \bigcup_{k=0}^\infty \bigcup \mathcal B_k.
\end{equation*}
We define an oscillating function $g^{(1)}\colon \R \to \{-1,0,1\}$ associated to the black intervals by
\begin{equation}
\label{e:g1}
g^{(1)} = \sum_{k \in S} (-1)^k \sum_{B_k \in \mathcal B_k} \chi_{B_k}.
\end{equation}
Here $\chi_{B_k} \colon \R \to \{0,1\}$ is the characteristic function of $B_k$.
The series \cref{e:g1} converges pointwise and in $L^1(\R)$ and takes values in $\{-1,0,1\}$ since the black intervals $B_k$ are pairwise disjoint.

From $E^{(1)}$ and $g^{(1)}$ we construct the corresponding $d$\=/dimensional objects $E \subseteq \R^d$ and $g\colon \R^d \to \{-1,0,1\}$ by taking a tensor product:
\begin{gather*}
E
=
\bigl\{(x_1,\ldots,x_d) \in [-\tfrac12,\tfrac12]^d
\mid x_i \in E^{(1)} \text{ for some $i$}
\bigr\},
\\
g(x_1,\ldots,x_d)
=
g^{(1)}(x_1) \cdots g^{(1)}(x_d).
\end{gather*}
We next use standard techniques to show that $E$ has full Hausdorff dimension precisely when \cref{e:gmean} holds.
The reader is advised to skip the proof of this result on a first reading.
\begin{lemma}
The Hausdorff dimension of $E$ is $d$ if and only if $\Theta_k^{1/k} \to 1/2$ as $k \to \infty$.
\end{lemma}
\begin{proof}
From the above recursive construction one can verify that the lengths of any $W_k \in \mathcal W_k$ and $B_k \in \mathcal B_k$ are
\begin{equation}
\label{e:len}
|W_k| = \Theta_k
\quad
\text{and}
\quad
|B_k| = (1-2\theta_k)\Theta_k.
\end{equation}
The families $\mathcal W_k$ and $\mathcal B_k$ each contain $2^k$ intervals.

Let $\diam(U) = \sup_{x,y \in U} |x-y|$ be the diameter of a set $U \subseteq \R^d$.
By definition, the Hausdorff dimension of $E$ is $d$ if and only if for every $\epsilon > 0$ there exists an $M = M(\epsilon) > 0$ such that for any open covering $E \subseteq \bigcup_{\gamma=0}^\infty U_\gamma$ we have
\begin{equation}
\label{e:hineq}
\sum_{\gamma=0}^\infty
\diam(U_\gamma)^{d-\epsilon}
\geq
M(\epsilon).
\end{equation}

First assume that $E$ has Hausdorff dimension $d$.
Since $E$ is covered by $d$ isometric copies of $\bigcup_{W_k \in \mathcal W_k} W_k \times [-1/2,1/2]^{d-1}$ and $\diam(W_k) = \Theta_k$, we can cover $E$ by $2^{k} \Theta_k^{-d+1} d$ boxes of diameter comparable to $\Theta_k$ and hence
\[
0<M(\epsilon)
\lesssim
2^k \Theta_k^{-d+1} d \cdot \Theta_k^{d-\epsilon}
=
2^k \Theta_k^{1-\epsilon} d
\]
for any $\epsilon > 0$.
Taking $k$th roots and letting $\epsilon \to 0$, this implies
\[
\limsup_{k \to \infty} \Theta_k^{1/k} \geq \tfrac12.
\]
Since $\Theta_k = \theta_0 \cdots \theta_{k-1} < 2^{-k}$, we in fact have $\Theta_k^{1/k} \to 1/2$ as $k \to \infty$.

Now we assume that $\Theta_k^{1/k} \to 1/2$ as $k \to \infty$.
Let $\mu$ be the Borel probability measure supported in $E$ given by
\[
\mu(W_k \times T) = 2^{-k} |T|
\]
for any $W_k \in \mathcal W_k$, $k \geq 0$, and any Borel set $T \subseteq [-1/2,1/2]^{d-1}$.

Fix an $\epsilon > 0$ and choose a large integer $N=N(\epsilon)$ such that
\[
\Theta_k^{1/k} \geq 2^{-1/(1-\epsilon)} \quad \text{for any $k \geq N$.}
\]
Let $U \subseteq \R^d$ be a Borel set of diameter $\diam(U) < \Theta_N$.
Since $\Theta_k \to 0$ as $k \to \infty$, there is a $k \geq N$ such that
\[
\Theta_{k+1} \leq \operatorname{diam}(U) < \Theta_k.
\]
Then $U$ intersects $W_k \times [-1/2,1/2]^{d-1}$ for at most two intervals $W_k \in \mathcal W_k$.
Hence,
\[
\frac{\mu(U)}{(2\diam(U))^{d-1}}
\leq
2^{-k+1}
\leq
4 \Theta_{k+1}^{1-\epsilon}
\leq
4 \operatorname{diam}(U)^{1-\epsilon}.
\]
Now if $E \subseteq \bigcup_{\gamma=0}^\infty U_\gamma$ is an open covering with $\diam(U_\gamma) < \Theta_N$, then
\[
\sum_{\gamma=0}^\infty
\diam(U_\gamma)^{d-\epsilon}
\geq
2^{-d-1}
\sum_{\gamma=0}^\infty
\mu(U_\gamma)
\geq
2^{-d-1}
\mu(E)
=
2^{-d-1}.
\]
This shows \cref{e:hineq} with $M(\epsilon) = \min(2^{-d-1}, \Theta_{N(\epsilon)}^{d-\epsilon})$ for any covering.
Hence $E$ has Hausdorff dimension $d$.
\end{proof}

Note that $\check g$ is bounded since $g$ is integrable.
Hence, we can prove \cref{i:ft} of \cref{p:criterion} by showing the following result.
\begin{lemma}
Let $p>1$.
It holds that
\[
\|\check g\|_{L^p(\R^d)}^{p/d}
\lesssim
\int_0^\infty
\biggl(\sum_{k \in S}
\frac{2^k \Theta_k}
{1 + (1-2\theta_k) \Theta_k |\xi|}
\prod_{j = 0}^{k-1} |{\cos((1-\theta_j) \Theta_j \pi \xi)}|
\biggr)^p
\, d\xi.
\]
\end{lemma}
\begin{proof}
By Fubini's theorem we have
$\check g(\xi_1, \ldots, \xi_d)
=
\check g^{(1)}(\xi_1)
\cdots
\check g^{(1)}(\xi_d)
$
and therefore
\begin{equation}
\label{e:fttensor}
\|\check g\|_{L^p(\R^d)}^{p/d}
=
\int_{-\infty}^\infty |\check g^{(1)}(\xi)|^p \, d\xi.
\end{equation}
Since the series \cref{e:g1} converges in $L^1(\R)$, we have
\begin{align}
\label{e:ftg1}
\check g^{(1)}(\xi) 
&=
\sum_{k \in S}
(-1)^k
\frac{\sin(|B_k| \pi \xi)}
{\pi \xi}
\sum_{B_k \in \mathcal B_k} e^{2\pi i m(B_k) \xi}
\end{align}
and the outer sum converges uniformly.
The midpoints of the black (and the white) intervals are given by
\begin{equation*}
\{ m(B_k) \mid B_k \in \mathcal B_k \} = \biggl\{ \sum_{j=0}^{k-1} \sigma_j \frac{1-\theta_j}2 \Theta_j \biggm\vert \sigma_j \in \{-1,1\} \biggr\}.
\end{equation*}
We use this to rewrite the inner sum in \cref{e:ftg1} and then we apply the identity $e^{-i\alpha}+e^{i\alpha} = 2\cos(\alpha)$ to get
\begin{align*}
\sum_{B_k \in \mathcal B_k} e^{2\pi i m(B_k) \xi}
&= \prod_{j=0}^{k-1}
\bigl(
e^{-i (1-\theta_j) \Theta_j \pi \xi}
+
e^{i (1-\theta_j) \Theta_j \pi \xi}
\bigr)
=
2^k \prod_{j = 0}^{k-1}
\cos((1-\theta_j) \Theta_j \pi \xi).
\end{align*}
Combining this with \cref{e:ftg1}, \cref{e:len} and the elementary estimate $|{\sin(\eta)/\eta}| \lesssim (1+|\eta|)^{-1}$ gives the pointwise bound
\begin{align*}
|\check g^{(1)}(-\xi)|
=
|\check g^{(1)}(\xi)|
&\lesssim
\sum_{k \in S}
\frac{2^k \Theta_k}
{1+ (1-2\theta_k)\Theta_k |\xi|}
\prod_{j = 0}^{k-1}
|{\cos((1-\theta_j) \Theta_j \pi \xi)}|.
\end{align*}
In view of \cref{e:fttensor}, this completes the proof of the lemma.
\end{proof}

We now complete the proof of \cref{i:leb} in \cref{p:criterion} by showing the following result.

\begin{lemma}
\label{l:leb}
If there exists an $\epsilon > 0$ such that \cref{e:goodpairs} holds, then $E$ is the non\-/Lebesgue set of $g$.
\end{lemma}
Note that \cref{e:goodpairs} implies $\liminf_{k \to \infty} \theta_k < 1/2$ and hence by \cref{e:len}:
\[
|E|
\leq
d \cdot \lim_{k \to \infty} \sum_{W_k \in \mathcal W_k} |W_k| \cdot |[-\tfrac12,\tfrac12]|^{d-1}
=
d \cdot \lim_{k \to \infty} 2^k \Theta_k
=
0.
\]
By the Lebesgue differentiation theorem, this is a necessary condition for $E$ to be the non\-/Lebesgue set of a locally integrable function.

The key step in the proof of \cref{l:leb} is the following lower bound on the oscillation of averages of the one\-/dimensional function $g^{(1)}$:
\begin{lemma}
\label{l:nocauchy}
For any $k \geq 0$ and any $W_k, W_k' \in \mathcal W_k$ it holds that \[\fint_{W_k} g^{(1)} \, dx = \fint_{W_k'} g^{(1)} \, dx\]
and for any $k,k+1 \in S$ and any $W_k \in \mathcal W_k$ and $W_{k+1} \in \mathcal W_{k+1}$ it holds that
\begin{equation*}
\biggl| \fint_{W_k} g^{(1)} \, dx - \fint_{W_{k+1}} g^{(1)} \, dx \biggr| \geq 2 (1-2\theta_k)(1-2\theta_{k+1}).
\end{equation*}
\end{lemma}
Before proving this lemma, we show how it can be used to prove \cref{l:leb}.
\begin{proof}[Proof of \cref{l:leb}]
The function $g$ is constant in each of the connected components of the open set $\R^d \setminus E$.
Hence, every point of that set is a Lebesgue point.

It remains to show that there are no Lebesgue points in $E$.
By \cref{e:goodpairs,l:nocauchy}, we can find sequences of indices $k(a), k'(a) \in S$, $a \geq 0$, with $k'(a) = k(a)\pm1$ and $k(a) \to \infty$ as $a \to \infty$ such that
\begin{equation}
\label{e:kaprop}
\theta_{k(a)}, \theta_{k'(a)} > \epsilon,
\quad
\biggl|
\fint_{W_{k(a)}} g^{(1)} \, dx - \fint_{W_{k'(a)}} g^{(1)} \, dx
\biggr|
>
8\epsilon^2,
\quad
\biggl|
\fint_{W_{k(a)}} g^{(1)} \, dx
\biggr|
>
4\epsilon^2
\end{equation}
for any $W_{k(a)} \in \mathcal W_{k(a)}$ and $W_{k'(a)} \in \mathcal W_{k'(a)}$.

Fix a point $(x_1,\ldots,x_d) \in E$ and let $a \geq 0$ be so large that $g^{(1)}$ is identically equal to $1$ or identically equal to $-1$ in a $\Theta_{k(a)}$\=/neighborhood of any $x_i \not\in E^{(1)}$, $i=1,\ldots,d$.
For $x_i \not\in E^{(1)}$, let $W_{k(a)}^{i}$ be any interval of length $\Theta_{k(a)}$ containing $x_i$.
For $x_i \in E^{(1)}$, choose intervals $W_{k(a)}^i \in \mathcal W_{k(a)}$ and $W_{k'(a)}^i \in \mathcal W_{k'(a)}$ that contain $x_i$.
In both cases we have, by $g = \pm 1$ and the third inequality in \cref{e:kaprop}, respectively:
\begin{equation}
\label{e:avlarge}
\biggl|
\fint_{W_{k(a)}^i} g^{(1)} \, dx
\biggr|
>
4 \epsilon^2.
\end{equation}

Fix an index $j$ for which $x_j \in E$ and consider the Cartesian products
\[
Q_{2a} = \prod_{i=1}^d W_{k(a)}^i,
\qquad
Q_{2a+1} = \prod_{i=1}^{j-1} W_{k(a)}^i \times W_{k'(a)}^j \times \prod_{i=j+1}^d W_{k(a)}^i.
\]
Making use of the tensor product structure of $g$, we can use the second inequality in \cref{e:kaprop} and \cref{e:avlarge} as follows:
\begin{align*}
\biggl|
\fint_{Q_{2a}} g
-
\fint_{Q_{2a+1}} g
\biggr|
&=
\biggl|
\fint_{W_{k(a)}^j} g^{(1)}
-
\fint_{W_{k'(a)}^j} g^{(1)}
\biggr|
\cdot
\prod_{i \neq j}
{
\biggl|
\fint_{W_{k(a)}^i} g^{(1)}
\biggr|}
\\
&>
2^{2d+1} \epsilon^{2d}
>0
\end{align*}
which holds for any large enough $a$, i.e.\ the averages of $g$ over the boxes $Q_b$ do not converge as $b \to \infty$.
On the other hand, by the first inequality in \cref{e:kaprop} these boxes have bounded eccentricity.
This shows that $(x_1,\ldots,x_d) \in \bigcap_{b \geq 0} Q_b$ is not a Lebesgue point of $g$.
\end{proof}
To complete the proof of \cref{l:leb} and hence of \cref{p:criterion}, we need to perform the calculations leading to \cref{l:nocauchy}.
\begin{proof}[Proof of \cref{l:nocauchy}]
By construction, it holds that $B_j \cap W_k = \emptyset$ when $B_j \in \mathcal B_j$, $W_k \in \mathcal W_k$ and $j < k$.
For $j \geq k$ we either have $B_j \cap W_k = \emptyset$ or $B_j \subseteq W_k$.
Hence,
\begin{align}
\label{e:avexp}
\fint_{W_k} g^{(1)} \, dx
&= \fint_{W_k} \sum_{j \in S;\,j \geq k} (-1)^j \sum_{B_j \in \mathcal B_j;\,B_j \subseteq W_k} \chi_{B_j}
\end{align}
For a fixed $W_k \in \mathcal W_k$, there are precisely $2^{j-k}$ intervals $B_j \in \mathcal B_j$ for which $B_j \subseteq W_k$.
Furthermore for fixed $k$ and $j$, the average $\fint_{W_k} \chi_{B_j} \, dx$ does by \cref{e:len} not depend on the choice of $W_k \in \mathcal W_k$ and $B_j \in \mathcal B_j$ as long as $B_j \subseteq W_k$.
This shows the first claim of the lemma.

It remains to prove the inequality in the second claim.
Let $k, k+1 \in S$ and let $W_k \in \mathcal W_k$ and $W_{k+1} \in \mathcal W_{k+1}$.
As $W_k$ is the disjoint union of two intervals in $\mathcal W_{k+1}$ and one interval in $\mathcal B_k$, we have by \cref{e:len}:
\[
\fint_{W_k} g^{(1)} \, dx
= 2 \theta_k \fint_{W_{k+1}} g^{(1)} \, dx
+ (-1)^k (1-2 \theta_k).
\]
Since $|g^{(1)}(x)| \leq 1$ for any $x \in \R$, it follows from ignoring all but the first term of the outer sum on the right\-/hand side of \cref{e:avexp} that
\[
(-1)^{k+1} \fint_{W_{k+1}} g^{(1)} \, dx
\geq -1 + 2(1-2\theta_{k+1}).
\]
Together with the last equation this implies
\[
(-1)^k
\biggl(
\fint_{W_k} g^{(1)} \, dx
-
\fint_{W_{k+1}} g^{(1)} \, dx
\biggr)
\geq
2(1-2\theta_k)(1-2\theta_{k+1}).
\]
Because the right\-/hand side is positive, this completes the proof of \cref{l:nocauchy}.
\end{proof}

We have now proved \cref{p:criterion}.

\begin{remark}
In the definition \cref{e:g1} of $g^{(1)}$, the oscillating coefficients $\pm1$ may be replaced by $0$ and $1$, respectively, yielding $\{0,1\}$\=/valued functions $g^{(1)}$ and $g$ instead of $\{-1,0,1\}$\=/valued ones.
The thus modified function $g$ still satisfies all properties that are asserted in \cref{p:criterion}.
However, if $d \geq 2$, then the non\-/Lebesgue set of $g$, while still of full Hausdorff dimension, would be a proper subset of $E$ since \cref{e:avlarge} would fail for some $x_i \not\in E^{(1)}$.
\end{remark}

\section{Incomplete cosine expansions of \texorpdfstring{$\sin(x)/x$}{sin(x)/x}}
\label{s:sinc}

In this section, we bound $p$\=/integrals of products of cosines with dyadic phases.
Our motivation is Euler's product expansion of the $\operatorname{sinc}$ function:
\begin{equation}
\label{e:euler}
\prod_{j=1}^\infty \cos(2^{-j} \xi) = \frac{\sin(\xi)}{\xi}.
\end{equation}
A quick proof of this identity can be obtained by iterating the double\-/angle formula $\sin(\xi) = 2 \sin(\xi/2) \cos(\xi/2)$ and using that $2^n \sin(\xi/2^n) \to \xi$ as $n \to \infty$.
Other proofs are possible, see e.g.\ the probabilistic proof in \cite{Kac59}.

The function in \eqref{e:euler} lies in $L^p_\xi(\R)$ for any $p > 1$.
We are interested in the stability of this property under omission of factors from the product of cosines.
First, it follows from \eqref{e:euler} that the product can be truncated after logarithmically in $|\xi|$ many steps without a loss in the decay rate.
More precisely,
\begin{align}
\label{e:eulertrunc}
\Bigl|\frac{\sin(\xi)}{\xi}\Bigr|
\leq
\prod_{j = 1}^n |{\cos(2^{-j} \xi)}|
\leq
\frac\pi2 \cdot \Bigl|\frac{\sin(\xi)}{\xi}\Bigr|
\qquad
\text{if $|\xi| \leq 2^{n-1} \pi$.}
\end{align}
However, if any further factor is omitted from this finite product, then the pointwise upper bound fails dramatically for some $|\xi| \leq 2^{n-1} \pi$.
We will therefore focus on integral estimates.
Given a finite set $J$ of integers, we define its \emph{number of components} $b(J)$ as follows:
\begin{equation*}
b(J) = \# \{j \in J \mid j-1 \not\in J\}.
\end{equation*}

\begin{lemma}
\label{l:eulergap}
Let $n \geq 1$ be an integer and let $J \subseteq \{1,2,\ldots,n\}$.
Then it holds for every $p \in (1,\infty)$ that
\begin{align}
\label{e:eulergap}
\int_0^{2^{n-1} \pi}
\prod_{j \in J}
|{\cos(2^{-j} \xi)}|^p
\, d\xi
&\leq
2^{n-|J|-1} \pi
C_p^{b(J)}
\end{align}
where $C_p$ is a finite constant that depends only on $p$.
\end{lemma}

\begin{proof}
We prove the lemma with the (non\-/optimal) constant
\begin{equation*}
C_p = 2 \pi^p \sum_{s=1}^\infty \frac1{s^p} < \infty.
\end{equation*}
We proceed by induction on the number of components $b(J)$.
If $b(J)=0$, then $J$ is empty and \cref{e:eulergap} is immediate.

Now fix a nonnegative integer $b$ and assume that \cref{e:eulergap} holds whenever $b(J) = b$.
Let $n$ be a positive integer and let $J_1$ be a subset of $\{1, \ldots, n\}$ such that $b(J_1) = b+1$.
We can decompose this set as $J_1 = J_0 \cup \{\ell, \ell+1, \ldots, m\}$ where $b(J_0) = b$ and $\sup(J_0) + 2 \leq \ell \leq m \leq n$.
Write $n_0 = \max(J_0 \cup \{0\})$.

We need to show \cref{e:eulergap} for $J_1$.
To this end, we cover the domain of integration $[0,2^{n-1}\pi]$ by the essentially disjoint intervals of equal length
\begin{equation*}
A(q,r)
=
[(2^{m-1} q + 2^{n_0-1} r) \pi, (2^{m-1} q + 2^{n_0-1} (r + 1)) \pi]
\end{equation*}
for any integers $q$ and $r$ with $0 \leq q < 2^{n-m}$ and $0 \leq r < 2^{m-n_0}$.
If $j \in J_0$, then the function $\xi \mapsto |{\cos(2^{-j} \xi)}|$ is even and $2^{n_0}\pi$\=/periodic.
We use this and the induction hypothesis to obtain
\begin{align*}
\int_{A(q,r)}
\prod_{j \in J_0}
|{\cos(2^{-j} \xi)}|^p
\, d\xi
&=
\int_0^{2^{n_0-1}\pi}
\prod_{j \in J_0}
|{\cos(2^{-j} \xi)}|^p
\, d\xi
\\
&\leq
2^{n_0-|J_0|-1} \pi C_p^{b(J_0)}
\\
&=
2^{n_0-|J|+m-\ell} \pi C_p^{b(J)-1}.
\end{align*}
Similarly for $j \leq m$, the function $\xi \mapsto |{\cos(2^{-j} \xi)}|$ is even and $2^m \pi$\=/periodic.
This gives
\begin{align*}
\sup_{\xi \in A(q,r)} \prod_{j=\ell}^m |{\cos(2^{-j} \xi)}|
&=
\sup_{\xi \in A(0,r)} \prod_{j=\ell}^m |{\cos(2^{-j} \xi)}|
\notag
\\
&=
\sup_{\xi \in A(0,r)} \prod_{j=1}^{m-\ell+1} |{\cos(2^{-j} 2^{1-\ell} \xi)}|
\notag
\\
&\leq
\frac \pi {1 + 2^{n_0-\ell} \pi r}.
\label{e:slow}
\end{align*}
The last inequality follows from \cref{e:eulertrunc} and the definition of $A(0,r)$.
We combine the last two estimates to get an estimate for the product over the full set of indices $J_1$:
\[
\int_{A(q,r)}
\prod_{j \in J_1}
|{\cos(2^{-j} \xi)}|^p
\, d\xi
\leq
\frac
{2^{n_0-|J|+m-\ell} \pi^{p+1} C_p^{b(J)-1}}
{(1 + 2^{n_0-\ell} \pi r)^p}.
\]
Note that the numerator does not depend on $q$ or $r$ and the denominator does not depend on $q$.
Therefore, we can sum over $q$ and $r$ as follows:
\begin{align*}
\int_0^{2^{n-1} \pi}
\prod_{j \in J_1}
|{\cos(2^{-j} \xi)}|^p
\, d\xi
&=
\sum_{q=0}^{2^{n-m}-1}
\sum_{r=0}^{2^{m-n_0}-1}
\int_{A(q,r)}
\prod_{j \in J_1}
|{\cos(2^{-j} \xi)}|^p
\, d\xi
\\
&\leq
2^{n-m}
2^{n_0-|J|+m-\ell} \pi^{p+1} C_p^{b(J)-1}
\sum_{r=0}^\infty
\frac1{(1 + 2^{n_0-\ell} \pi r)^p}
\\
&\leq
2^{n-|J|} \pi^{p+1} C_p^{b(J)-1}
\sum_{s=1}^\infty
\frac1{s^p}.
\end{align*}
For the last inequality, we replaced $\pi r$ by the largest multiple of $2^{\ell-n_0}$ not exceeding~$r$.
By our choice of $C_p$, this shows \cref{e:eulergap} for $J_1$ and hence closes the induction.
\end{proof}

In the proof of \cref{t} we will need the following perturbed version of the previous result.

\begin{lemma}
\label{l:pert}
Let $n \geq 1$ be an integer, let $J \subseteq \{1,2,\ldots,n\}$, let $1 < p \leq p_0 < \infty$ and let $\epsilon > 0$.
There exists a number $\delta = \delta(n, p_0, \epsilon) > 0$ that does not depend on $J$ or $p$ such that if
\begin{equation}
\label{e:phiclose}
\phi_j \in (2^{-j}(1-\delta), 2^{-j}(1+\delta))
\end{equation}
for all $j \in J$, then the following inequality holds:
\begin{equation*}
\int_0^{2^{n-1}\pi}
\prod_{j \in J}
|{\cos(\phi_j \xi)}|^p
\, d\xi
\leq
(1+\epsilon)
2^{n-|J|-1} \pi C_p^{b(J)}.
\end{equation*}
\end{lemma}
\begin{proof}
First, let $J \subseteq \{1,\ldots,n\}$ be fixed.
Consider the integrals
\[
I^p(\{\phi_j\}_{j \in J})
=
\int_0^{2^{n-1}\pi}
\prod_{j \in J}
|{\cos(\phi_j \xi)}|^p
\, d\xi
\in (0, \infty).
\]
By the dominated convergence theorem, $I^p(\{\phi_j\}_{j \in J})$ is continuous in $p \in [1,\infty)$ and $\phi_j \in \R$.
Hence it is uniformly continuous once $p$ and the $\phi_j$ are confined to a compact domain.
We have by compactness that
\[
\int_0^{2^{n-1} \pi}
\prod_{j \in J}
|{\cos(2^{-j} \xi)}|^p
\, d\xi
> 0
\]
uniformly in $p \in [1,p_0]$.
Now this together with uniform continuity allows us to find $\delta=\delta(n,J,p_0,\epsilon) > 0$ such that
\[
\int_0^{2^{n-1}\pi}
\prod_{j \in J}
|{\cos(\phi_j \xi)}|^p
\, d\xi
\leq
(1+\epsilon)
\int_0^{2^{n-1}\pi}
\prod_{j \in J}
|{\cos(2^{-j} \xi)}|^p
\, d\xi
\]
whenever $1 \leq p \leq p_0$ and \cref{e:phiclose} holds.
As there are only finitely many subsets of $\{1,\ldots,n\}$, the number $\delta$ can in fact be chosen independently of $J$.
An application of \cref{l:eulergap} completes the proof.
\end{proof}

\section{Proof of Theorem~\ref{t}}
\label{s:proof}

In this section, we use the criteria of \cref{p:criterion} and the analytical \cref{l:pert} to prove \cref{t} in five steps.

\subsection{Choice of parameters}
We fix a set $S$ of nonnegative integers and a weight function $w: S \to \{2,3,4,\ldots\}$ and define
\[
w^+(k) = \sum_{j \in S;\,j<k} w(j)
\qquad \text{and} \qquad
\chi^+(k) = \#(S \cap [0,k-1])
\]
such that there is a finite constant $M \geq 2$ and there are infinitely many $k \in S$ for which $k+1 \in S$ and $w(k), w(k+1) \leq M$ and such that we have the following asymptotics:
\begin{equation}
\lim_{\substack{k \to \infty\\k \in S}} \frac{w(k)}{w^+(k)} = 0
\quad \text{and} \quad
\lim_{k \to \infty} \frac{\chi^+(k)}{w^+(k)} = 0
\quad \text{and} \quad
\lim_{k \to \infty} \frac{w^+(k)}k = 0.
\label{e:w}
\end{equation}
For example, we may choose $S = \{n^3 \mid n \geq 2\} \cup \{n^6+1 \mid n \geq 2\}$ and $w(n^3)=n$ if $n$ is not a square and $w(n^6)=w(n^6+1)=2$ for any $n \geq 2$.

Next we fix a function $r: \{0,1,2,\ldots\} \to \{0,1,2,\ldots\}$ satisfying
\begin{equation}
\lim_{s \to \infty} r(s) = \infty
\quad \text{and} \quad
\lim_{s \to \infty}
\frac{r(s)}{w^+(s)}
=
0.
\label{e:r}
\end{equation}
Let $\delta(n,p_0,\epsilon)$ be the numbers from \cref{l:pert} and write $\delta(n) = \delta(n,2,1)$.
We may assume that $0 < \delta(n+1) < \delta(n) < 1/2$ for any $n$.
Finally we choose dissection ratios $\theta_j \in (0,1/2)$ as follows:
\begin{equation*}
\theta_j
=
\begin{cases*}
2^{-w(j)} & if $j \in S$, \\
\frac12(1-\alpha_j) & if $j \not\in S$,
\end{cases*}
\end{equation*}
with error terms $\alpha_j$ satisfying $0 < 2\alpha_{j+1} \leq \alpha_j \leq 1/4$ and
\begin{equation}
\label{e:alphadelta}
2 \alpha_j
\leq
\inf\{\delta(2s) \mid s \geq 0, r(s) \leq j\}
\end{equation}
for any $j \geq 0$.
Note that the infimum above is positive since $r(s) \to \infty$ as $s \to \infty$.

\subsection{Asymptotics of products of dissection rates}
It follows that
\begin{equation*}
\prod_{j=k}^\infty
(1-\alpha_j)
\geq
1 - \sum_{j=k}^\infty \alpha_j
\geq
1 - 2\alpha_k.
\end{equation*}
This bound is significant because of the following expansion of $\Theta_k$ based on our choice of dissection ratios:
\begin{align*}
\Theta_k
&=
\prod_{j \in S;\,j < k} 2^{-w(j)}
\cdot
\prod_{j \not\in S;\,j < k} \tfrac12(1-\alpha_j)
\\
&=
2^{-k-w^+(k)+\chi^+(k)}
\cdot
\prod_{j \not\in S;\,j < k} (1-\alpha_j).
\end{align*}
For notational convenience, we define corresponding to any index $i$ a larger index by
\[
i^\ast = i + w^+(i) - \chi^+(i).
\]
Then, we obtain the following inequalities for any $k \geq 0$ and $j \geq r(s)$:
\begin{align}
\label{e:Thetabd}
2^{-k^\ast-1}
\leq{}
&\Theta_k
\leq
2^{-k^\ast},
\\
\label{e:partialThetabd}
(1- \delta(2s))
2^{-j^\ast + r(s)^\ast}
\leq{}
&\Theta_{r(s)}^{-1} \Theta_j
\leq
2^{-j^\ast + r(s)^\ast}.
\end{align}
Hence, $\Theta_k$ is close to a particular power of $1/2$ and this relation is even tighter for the partial product $\Theta_{r(s)}^{-1} \Theta_j$.
Using the last two limits in \cref{e:w}, it follows from \cref{e:Thetabd} that $\Theta_k^{1/k} \to 1/2$ as $k \to \infty$.
This verifies the assumption \cref{e:gmean} of \cref{p:criterion}.
Notice that the assumption \cref{e:goodpairs} is satisfied for any $\epsilon \in (0,2^{-M})$ by our choices of $S$ and $\theta_j$.

\subsection{Integral decomposition}
In order to prove \cref{t}, it remains to verify the inequality in \cref{i:ft} of \cref{p:criterion} for any $p \in (1,2]$.
Since $1-2\theta_k \geq 1/2$ for $k \in S$ we may omit this term from the left\-/hand side of that inequality.
Hence, in order to prove \cref{t} it now suffices to show that
\[
H(\xi)
=
\sum_{k \in S}
\frac{2^k \Theta_k}
{1+\Theta_k |\xi|}
\prod_{j=0}^{k-1}
|{\cos((1-\theta_j)\Theta_j\pi\xi)}|
\]
lies in $L^p_\xi([0,\infty))$ for any $p \in (1,2]$.
Consider the integrals at scales $s \geq k$:
\begin{equation}
I^p_{k,s}
=
\int_0^{\Theta_s^{-1}}
\prod_{j=0}^{k-1} 
|{\cos((1-\theta_j)\Theta_j\pi\xi)}|^p
\, d\xi.
\label{e:Ip}
\end{equation}
We use Minkowski's inequality in $L^p([0,\infty))$ and for every $k \in S$ we decompose $[0,\infty)$ into the subintervals $[0,\Theta_k^{-1})$ and $[\Theta_s^{-1}, \Theta_{s+1}^{-1})$ for $s \geq k$ to obtain the bound
\begin{equation}
\|H\|_{L^p([0,\infty))}
\lesssim
\sum_{k \in S}
2^k \Theta_k
\biggl(
I^p_{k,k}
+
\sum_{s=k}^\infty
\frac{\Theta_s^p}{\Theta_k^p}
I^p_{k,s+1}\biggr)^{1/p}
\eqqcolon
\sum_{k \in S} A_k^p.
\label{e:mink}
\end{equation}

We let go of some factors in the product of cosines in \eqref{e:Ip}, perform a linear change of variables and then slightly enlarge the domain of integration using \cref{e:Thetabd} to obtain
\begin{align}
I^p_{k,s}
&\leq
\pi^{-1}\Theta_{r(s)}^{-1}
\int_0^{\Theta_{r(s)} \Theta_s^{-1}\pi}
\prod_{\substack{j \not\in S\\r(s)\leq j \leq k-1}}
|{\cos((1-\theta_j)\Theta_{r(s)}^{-1}\Theta_j\xi)}|^p
\, d\xi
\notag
\\
&\lesssim
\Theta_{r(s)}^{-1}
\int_0^{2^{s^\ast - r(s)^\ast + 1} \pi}
\prod_{\substack{j \not\in S\\r(s)\leq j \leq k-1}}
|{\cos((1-\theta_j)\Theta_{r(s)}^{-1}\Theta_j\xi)}|^p
\, d\xi.
\label{e:Ipbd}
\end{align}

\subsection{Application of Lemma~\ref{l:pert}}
We next analyze the phases of the cosines in \cref{e:Ipbd}.
If $j \not\in S$ and $r(s) \leq j \leq k-1 < s$, then \cref{e:partialThetabd}, \cref{e:alphadelta} and the inequality $1-\theta_j > 1/2$ imply
\begin{align*}
|2^{-(j^\ast-r(s)^\ast+1)}
-
(1-\theta_j)\Theta_{r(s)}^{-1}\Theta_j|
\leq
2^{-(j^\ast-r(s)^\ast+1)}
\delta(2s).
\end{align*}
If $k$ and hence $s$ are larger than some sufficiently large constant $K = K_{S,\theta_j}$, then we have by \cref{e:w} that $s^\ast \leq 2s-2$ and $j^\ast-r(s)^\ast+1 \leq 2s$.
Therefore, we verified the assumption \cref{e:phiclose} in \cref{l:pert} in the case when $k \geq K$ with the near\-/dyadic phases
\[
\phi_{j^\ast-r(s)^\ast+1}
=
\phi_{j^\ast-r(s)^\ast+1, s}
=
(1-\theta_j)\Theta_{r(s)}^{-1}\Theta_j
\]
and the following set of dyadic exponents:
\[
J = J_{k,s}
=
\{
j^\ast - r(s)^\ast + 1
\mid
j \not\in S \text{ and } r(s) \leq j \leq k-1
\}.
\]
Since the map $j \mapsto j^\ast$ is strictly increasing and therefore injective we have
\[
\#J_{k,s}
\geq
k - \chi^+(k) - r(s)
=
k^\ast - w^+(k) - r(s).
\]
We have $(j+1)^\ast = j^\ast + 1$ if and only if $j \not\in S$.
Hence for any $j \notin S$ with $r(s) < j \leq k-1$ the condition $j^\ast - r(s)^\ast \not\in J_{k,s}$ is equivalent to $j-1 \in S$.
This allows us to bound the number of components of $J_{k,s}$:
\[
b(J_{k,s})
\leq
\#(S \cap [r(s), k-2]) + 1
\leq
\chi^+(k) + 1.
\]
Furthermore, the set $J_{k,s}$ is bounded from above:
\[
\sup(J_{k,s})
\leq
(k-1)^\ast - r(s)^\ast + 1
\leq
s^\ast - r(s)^\ast.
\]
Compare this to the upper bound of integration in \cref{e:Ipbd} to see that \cref{l:pert} is applicable to the integral in \cref{e:Ipbd}.
We obtain that
\begin{align*}
I^p_{k,s}
&\lesssim
\Theta_{r(s)}^{-1}
2^{s^\ast-r(s)^\ast-\# J_{k,s}}
C_p^{b(J_{k,s})},
\quad
\text{if $k \geq K$.}
\end{align*}
We use \cref{e:Thetabd} and the bounds above on $\# J_{k,s}$ and $b(J_{k,s})$ to bring this estimate into a more convenient form:
\begin{align*}
I^p_{k,s}
&\lesssim
\Theta_s^{-1}
2^{-\# J_{k,s}}
C_p^{b(J_{k,s})}
\lesssim
\Theta_s^{-1} \Theta_k 2^{w^+(k)+r(s)} C_p^{\chi^+(k)+1},
\quad
\text{if $k \geq K$.}
\end{align*}

\subsection{Conclusion}
We use the previous inequality and \cref{e:Thetabd} to estimate the terms $A_k^p$, $k \in S$, of the sum in \cref{e:mink}:
\begin{equation*}
A_k^p
\lesssim
2^{-(1-1/p)w^+(k)+\chi^+(k)}
C_p^{(\chi^+(k)+1)/p}
\biggl(
2^{r(k)}
+
\sum_{s=k}^\infty
\frac{\Theta_s^{p-1}}
{\Theta_k^{p-1}}
\theta_s^{-1}
2^{r(s+1)}
\biggr)^{1/p}.
\end{equation*}
Fix a positive number $\epsilon$ such that $2\epsilon < 1-1/p$ and $\epsilon < p-1$.
After possibly increasing $K$, we obtain from the limits \cref{e:w,e:r} that
\[
\theta_s^{-1}
2^{r(s+1)}
\leq
2^{\epsilon w^+(s)}
=
2^{\epsilon w^+(k)}
2^{\epsilon (w^+(s)-w^+(k))}
\leq
2^{\epsilon w^+(k)}
\Theta_s^{-\epsilon}\Theta_k^{\epsilon},
\quad
\text{if $s \geq K$}
\]
and further for $k \in S$ with $k \geq K$:
\begin{align*}
A_k^p
&\lesssim
2^{-(1-1/p-2\epsilon)w^+(k)}
\biggl(
1
+
\sum_{s=k}^\infty
\frac{\Theta_s^{p-1-\epsilon}}
{\Theta_k^{p-1-\epsilon}}
\biggr)^{1/p}
\leq
2^{-(1-1/p-2\epsilon)w^+(k)}
S_{p,\epsilon}
\end{align*}
where $S_{p,\epsilon}$ is a constant depending on $p$ and $\epsilon$.
We split the sum in \cref{e:mink} as follows:
\[
\|H\|_{L^p([0,\infty))}
\lesssim
\sum_{k \in S;\,k < K} A_k^p
+
\sum_{k \in S;\,k \geq K} A_k^p.
\]
The sum over $k \geq K$ is dominated by a convergent geometric series due to the last bound on $A_k^p$ and since $1-1/p-2\epsilon > 0$ and $w^+(k_1) \leq w^+(k_2)-2$ for any $k_1, k_2 \in S$ with $k_1 < k_2$.
The sum over $k < K$ above is finite since it is the sum of finitely many terms $A_k^p$, each of which is finite.
Hence $H$ lies in $L^p([0,\infty))$.
This completes the proof of \cref{t}.

\section{Proofs of Corollaries~\ref{c:endpt} and~\ref{c:full}}
\label{s:proofcset}

We first prove \cref{c:endpt} using \cref{t,t:kovac}.
Based on this, we then we prove \cref{c:full}.

\begin{proof}[Proof of Corollary~\ref{c:endpt}]
By \cref{t} there is a function $f \in \bigcap_{p \in (1,\infty]} L^p(\R^d)$ such that the non\-/Lebesgue set $E$ of $\hat f$ is compact and has full Hausdorff dimension~$d$.
We will show that $E$ satisfies the remaining assertion of \cref{c:endpt}.

To this end, let $\mu$ be a nonzero Borel measure such that \cref{e:res} holds for some exponents $p \in (1,2]$ and $q \in [1,\infty]$.
It suffices to show that $\mu(E)=0$.
A scaling argument shows that since $p > 1$ we necessarily have $q < \infty$.
Therefore, it follows from \cref{e:res} that $\mu$ is $\sigma$\=/finite and an interpolation with the trivial $L^1(\R^d) \to L^\infty(\mu)$ bound gives a $L^{p_1}(\R^d) \to L^{q_1}(\mu)$ restriction estimate with $1 < p_1 < q_1 < \infty$.
Hence, by \cref{t:kovac} and since $f$ lies in $L^{2p_1/(p_1+1)}(\R^d)$, $\mu$\=/a.e.\ point is a Lebesgue point of $\hat f$.
But $\hat f$ has no Lebesgue points in $E$ and therefore $\mu(E)=0$.
This completes the proof.
\end{proof}

\begin{proof}[Proof of Corollary~\ref{c:full}]
If $p=2$, then $\alpha = d$ and any compact set of positive Lebesgue measure proves the corollary.
We may now assume that $p<2$.

Let $E$ be the set from \cref{c:endpt}.
Let $E_1$ be a compact subset of $E$ of Hausdorff dimension equal to $\alpha$.
Such a subset can be explicitly obtained by reducing the Cantor set of \cref{s:osc} in an appropriate way, but its existence is also guaranteed by a theorem of Besicovitch~\cite{Bes52}, see also Davies~\cite{Dav52}.

If $p=1$, then we are finished since $\rexp(E) = 1$ and hence $\rexp(E_1) = 1$.
We may now assume that $1 < p \leq 2d/(2d-\alpha)$.
Hence, there is an $\alpha_0 \in (0,\alpha]$ such that $p = 2d/(2d-\alpha_0)$.
Using the previous reduction to the case $p<2$, we see that $\alpha_0<d$.
By \cite[Theorem~2]{LW18}, there is a compact set $E_2$ of Hausdorff dimension $\alpha_0$ such that $\rexp(E_2) = p$.
Now we have
\[
\dimH(E_1 \cup E_2)
=
\max(\dimH(E_1), \dimH(E_2))
=
\max(\alpha, \alpha_0)
=
\alpha
\]
and similarly
\[
\rexp(E_1 \cup E_2)
=
\max(\rexp(E_1), \rexp(E_2))
=
\max(1,p)
=
p.
\]
Hence, the compact set $E_1 \cup E_2$ has the claimed properties.
\end{proof}

\printbibliography
\end{document}